\documentclass[12pt]{amsart}
\usepackage{amssymb}
\usepackage{graphicx}
\usepackage{epstopdf}
\usepackage{color}
\usepackage{overpic}
\input{epsf}

\theoremstyle{plain}
\newtheorem{theorem}{Theorem}[section]
\theoremstyle{definition}
\newtheorem{definition}[theorem]{Definition}
\newtheorem{proposition}[theorem]{Proposition}
\newtheorem{remark}[theorem]{Remark}

\newtheorem{example}[theorem]{Example}
\newtheorem{problem}{Problem}

\renewcommand{\r}{\mathbb{R}}
\newcommand{\z}{\mathbb{Z}}
\newcommand{\discr}{discr}

\usepackage[todonotes={textsize=footnotesize}]{changes}

\makeatletter\def\@captionfont{\normalfont\footnotesize}\makeatother

\title{Liftings of surfaces in the plane}

\author[Oleg Karpenkov,
Brigitte Servatius, Herman Servatius]{Oleg Karpenkov,
Brigitte Servatius, Herman Servatius}

\address{Oleg Karpenkov\\University of Liverpool
}
\email{karpenk@liverpool.ac.uk}

\address{Brigitte Servatius\\
Worcester Polytechnic Institute}
\email{bservat@wpi.edu}

\address{Herman Servatius\\
Worcester Polytechnic Institute}
\email{hservat@wpi.edu}

\thanks{
The collaborative research on this article was supported by the National Science Foundation under Grant No. DMS-1929284 while the authors were in residence at the Institute for Computational and Experimental Research in Mathematics in Providence (ICERM), RI, during the Geometry of Materials, Packings and Rigid Frameworks
Jan 29 - May 2, 2025 program.
}

\begin{document}

\begin{abstract}
In this note we provide a topological definition of Maxwell-Cremona liftings for non-planar frameworks of surfaces (both oriented and non-oriented). In the non-oriented case we give an estimate on the dimension of self-stresses, when the frameworks will posses a non-trivial lifting.
\end{abstract}

\maketitle
\tableofcontents

\section*{Introduction}

In this paper we study liftings of non-planar frameworks in topological terms. In order to construct such liftings we consider polygonal surfaces whose 1-skeletons coincide with the framework
The construction works for both oriented and non-oriented surfaces.

\vspace{2mm}

Liftings were introduced by J.~C.~Maxwell~\cite{maxwell1864xlv,maxwell1870reciprocal} in 1864 and further developed by L.~Cremona~\cite{cremona1872figure} in 1872.
Later in 1982 W.~Whiteley~\cite{Whiteley1982} proved that
every reciprocal diagram and every lifting corresponds to an equilibrium stress (see also~\cite{CW1993}).
Since then liftings were used in various branches of mathematics including rigidity theory, discrete and computational geometry, and polyhedral combinatorics.
In particular, the solution to the Carpenter's Rule problem~\cite{Con03} and the proof of Steinitz's Theorem~\cite{richter2006realization} in polyhedral combinatorics substantially use liftings (see~\cite{EadesGarvan} and~\cite{Pak} for further details).
In 2014 C.~Borcea and I.~Streinu discussed liftings for planar periodic frameworks, see~\cite{borcea2014liftings,BorStr}.
The link to periodic cubic irrationalities was developed in~\cite{KFMS2024}.
Reciprocal diagrams of toroidal frameworks are studied by J.~Erickson and P.~Lin in ~\cite{ErLin}.
Finally, we would like to mention a generalization of liftings to an arbitrary dimension constructed in~\cite{KFMS2023}.

\vspace{2mm}

The liftings for non-planar realizations of graphs were not known till recent. The first extension of the definition of liftings to the non-planar case was obtained by Z.~Cooperband, R.~Ghrist and J.~Hansen in a preprint 2023~\cite{cooper}. They describe liftings for planar frameworks using cosheaf homology, which extend the definition to liftings to oriented surfaces.
In particualr they have proven that if the dimension of the space of self-stresses for a graph corresponding to an oriented surface of genus $g$ exceeds $6g$ then there exists a surface lifting of this graph.

\vspace{2mm}

In this paper we develop an alternative combinatorial approach to liftings. We develop an explicit formula to write down liftings for both oriented and non-oriented surfaces.
As a consequence we show that a non-oriented surface with holes $S$ has a lifting provided that the dimension of the space of self-stresses exceeds $3b_1(S)$, where $b_1(S)$ denotes the first  Betti number of $S$.

\vspace{2mm}

\noindent
{\bf Organization of the paper.}
We start in Section~\ref{Notions and definitions} with general definitions of polygonal surfaces and self-stresses
Further, in Section~\ref{Lifts for oriented face-paths}
we introduce the notion of lifts along face-paths in polyhedral surfaces in $\r^2$ defined by self-stresses in edges.
We prove that the lifts defined by two homotopic face-paths coincide.
In Section~\ref{Liftings of surfaces}
we introduce the notion of liftings for frameworks whose edges may self-intersect, and study its basic properties. In particular we provide an estimate on the dimension of self-stresses, when the frameworks posses non-trivial liftings (Theorem~\ref{betti-dimension}).

\section{Self-stresses and their liftings in the classical settings}
\label{Notions and definitions}

\subsection{Self-stressed frameworks}

\label{def:framework}
Consider an arbitrary  finite, connected graph $G = (V, E)$ without loops and multiple edges. Here
$V = \{v_1, \ldots, v_d\}$ is the set of
\emph{vertices} and $E$ is the set of \emph{edges}.
Denote by $v_i v_j$ the edge joining $v_i$ and $v_j$.

\begin{definition}
Let $G = (V, E)$ be a graph as above.

\begin{itemize}

\item
A \emph{framework} $G(p)$ in $\r^2$ is a $d$-tuple of points $(p_1,\ldots, p_n)\in(\r^2)^d$.
The points $p_i$ and $p_j$ are {\it connected by an edge} (i.e, by a straight line segment) if
$v_i v_j$ is an edge of $G$.

\item A \emph{stress} $w$ on a framework $G(p)$ is an assignment of real scalars $\omega_{ij}$ to its edges $p_i p_j$. Here we assume that $w_{ij} = w_{ji}$.

\item A stress $\omega$ is called a \emph{self-stress} if the
      equilibrium condition
      \begin{equation*}
        \sum\limits_{\{j \mid v_iv_j\in E\}} \omega_{ij} (p_i - p_j) = 0
      \end{equation*}
      is fulfilled at every vertex $i$.

    \item A pair $(G(p), w)$ is called a \emph{tensegrity} if $\omega$ is a self-stress  for the framework $G(p)$.
  \end{itemize}
\end{definition}

\subsection{A classical notion of lifting for planar frameworks}
A framework in $\r^n$ is planar two of its edges intersect if and only if they are adjacent to the same vertex, this vertex is the only point of intersection of such edges.

Let us recall the classic definition of liftings.
\begin{definition}
  \label{def:liftings}
  Consider a planar embedding $G(p)$ with straight edges in $\r^2$.
  A continuous and piecewise affine function $L: \r^2 \to \r$ is called a
  \emph{lifting} of $G(p)$ if its singular set is contained in the framework $G(p)$.
\end{definition}

In this paper we extend the notion of liftings to the case of non-planar frameworks and even to non-planar graphs.

\newpage

\section{Lifts for oriented face-paths of planar surfaces}
\label{Lifts for oriented face-paths}

\subsection{Polygonal surfaces}

We say that a polygon in $\r^n$ is {\it flat} if all its vertices are contained in one two-dimensional plane and its boundary does not have self-intersections. (The flat polygons are not assumed to be convex.)
A {\it polygonal surface $($with boundary$)$ in $\r^n$} is a mapping of  a CW-complex of a surface $($with boundary$)$ to $\r^n$, $n\ge 2$ such that each face corresponds to a flat polygon. 

\vspace{2mm}

We say that the framework defined by the 1-skeleton of a polygonal surface 
$S$ is the {\it framework} for $S$, we denote it by $F(S)$.

\begin{remark}
We work mostly with the case $n=2$. Note that self-intersections are allowed (so the surfaces are not really subsets of $\r^n$). Furthermore, we restrict ourselves to combinatorial types of compact surfaces whose boundaries are homeomorphic to a disjoint collections of circles.
\end{remark}


\subsection{Oriented face-paths}

In this subsection we define face paths and the orientation for them.

\begin{definition}
Let $S$ be a polygonal surface in $\r^2$ and let $N$ be a positive integer.
\begin{itemize}
\item
A {\it face-path}  $\gamma(F,E)$ of $S$ of length $N$ from a face $f_0$ to a face of $f_N$ is a sequence of $F=(f_0,f_1,\ldots,f_N)$ together with a sequence of edges $E=(p_1q_1,\ldots, p_Nq_N)$ such that for all admissible $i$ the edge $p_iq_i$ is a common edge of faces $f_i$ and $f_{i+1}$.

\vspace{1mm}

\item
In case if $f_0$ and $f_N$ are the same, we say that a face-path is a {\it face-loop}.
\end{itemize}
\end{definition}

Let us recall the following general definitions.

\begin{definition}\quad
\begin{itemize}
\item We say that a face is {\it oriented} if its edges are oriented and they form an oriented loop.

\vspace{1mm}

\item We say that the orientations of two faces $f_1$ and $f_2$ {\it agree} at their common edge $e$, if the oriented edge loops for $f_1$ and for $f_2$ pass $e$ in opposite directions.
\end{itemize}
\end{definition}

\begin{definition}
We say that a face-path $\gamma(F,E)$ is {\it oriented} if every one of its faces and and every one of its edges is oriented  such that
\begin{itemize}
\item the orientations of consecutive faces $f_i$ and $f_{i+1}$ agree at $p_iq_i$;

\vspace{1mm}

\item the orientation of $p_iq_i$ is induced by the orientation of $f_{i-1}$,
\end{itemize}
for all admissible $i$.
\end{definition}

\begin{remark}
For non-oriented surfaces the path may come several times to the same face $f$, possibly taking opposite orientations. We allow to do so, as we orient faces of the path rather than the faces of the surface itself.
\end{remark}




\subsection{Lifts along oriented face-paths}

Let us fix the following notation.

\begin{definition}
Let $S$ be a polygonal surface in $\r^2$, and let $\gamma$ be an oriented  face-path from $f_0$ to $f_1$ consisting of two faces with a common edge $p_1q_1$.
The {\it elementary lift} for the neighboring faces $f_i$ and $f_{i+1}$ is a linear function $L_{f_i}^{f_{i+1}}$ defined as follows:
$$
L_{f_i}^{f_{i+1}}: p \mapsto \det(p_1{-}q_1,p_1{-}p) \cdot w(p_iq_i)\cdot \rho_{\discr} (f_0,f_{1}),
$$
where
$$
\rho_{\discr}(f,g)=
\left\{
\begin{array}{ll}
1, & \hbox{if $f$ and $g$ are different faces of $S$,}\\
0, & \hbox{otherwise.}
\end{array}
\right.
$$

\end{definition}

We are ready to give the definition of a lift corresponding to an oriented face-path for a surface.

\begin{definition}
Let $S$ be a polygonal surface, and let $\gamma$ be an oriented face-path from $f_0$ to $f_N$.
Consider a framework $F(S)$ with a self-stress $w$ on it.
We say that the {\it lift}  from $f_0$ to $f_N$
with respect to $\gamma$ and $w$ is the following linear function:
$$
\sum\limits_{i=1}^N L_{f_i-1}^{f_{i}}.
$$
We denote it by $\tau_{\gamma,w}(f_0,f_N)$.
\end{definition}


\subsection{Homotopy invariance of lifts}

It turns out that the function $\tau_{\gamma,w}(f,f')$ does not depend on the choice of a path in the homotopy class of face-paths.

\begin{theorem}\label{homotopic paths}
The lift $\tau_{\gamma,w}(f,f')$ does not depend on the choice of a face-path $\gamma$ from $f$ to $f'$ in the class of homotopic oriented face-paths.
\end{theorem}

\begin{proof}
Consider two homotopic oriented face-paths. Then, by the definition, one of them can be taken to the other with a sequence of elementary path changes:

\begin{itemize}
\item {\bf First elementary move:} adding/removing two copies of a single face adjacent to the same $p_iq_i$:
$$
(f_1,\ldots, f_i,f_{i+1},\ldots f_N) \to (f_1,\ldots, f_i,g,g, f_{i+1},\ldots,f_N).
$$
Here $(f_i,g)$ and $(g,f_{i+1})$ both share the same edge $p_iq_i$.
The pair $(g,g)$ shares any edge $\hat p \hat q$ of $g$.

\vspace{2mm}

\item {\bf Second elementary move:} adding/removing a loop of faces around a vertex:
$$
(f_1,\ldots, f_i,f_{i+1},\ldots, f_N) \to (f_1,\ldots, f_i,g_1,\ldots, g_k, f_{i+1},\ldots, f_N).
$$
Where $(g_1,\ldots, g_k)$ is a loop of faces around either $p_i$ or $q_i$; here we assume that $g_1$ and $g_k$ share the edge $p_iq_i$.
The faces $g_s$ and $g_{s+1}$ shares the edge $p_i\hat q_s$
(or respectively $\hat p_s q_i$). Here we assume that $p_i\hat q_1=p_i q_i=p_i\hat q_k$ or $\hat p_1 q_i=p_i q_i=\hat p_k q_i$ respectively.

\end{itemize}

\vspace{2mm}

\noindent
{\it First elementary move invariance.}
We replace
$$
L_{f_i}^{f_{i+1}} \to
L_{f_i}^{g} + L_g^g+L_{g}^{f_{i+1}}.
$$
First of all $L_g^g=0$. Hence
$$
L_{f_i}^{g} + L_g^g+L_{g}^{f_{i+1}}=L_{f_i}^{g} +L_{g}^{f_{i+1}}.
$$
Both three faces $g$, $f_i$, and $f_{i+1}$ share the same edge. Hence at least two of them coincide. Direct computation in all four possible cases result in the equality:
$$
L_{f_i}^{f_{i+1}}=L_{f_i}^{g} + L_{g}^{f_{i+1}}.
$$

\vspace{2mm}

\noindent
{\it Second elementary move invariance.} For the second elementary move we replace
$$
L_{f_i}^{f_{i+1}} \to
L_{f_i}^{g_1}+\sum\limits_{j=1}^{k-1}L_{g_i}^{g_{i+1}}+
L_{g_k}^{f_{i+1}}.
$$

\vspace{2mm}

Let us perform the first elementary move and then collect the summands:
$$
\begin{aligned}
L_{f_i}^{g_1}+\sum\limits_{j=1}^{k-1}L_{g_i}^{g_{i+1}}+
L_{g_k}^{f_{i+1}}
&=\,
L_{f_i}^{g_1} +
\sum\limits_{j=1}^{k-1}L_{g_i}^{g_{i+1}}+L_{g_k}^{g_{1}}
+L_{g_1}^{g_{1}}+L_{g_1}^{f_{i+1}}
\\
&=\,
\big(L_{f_i}^{g_1} +L_{g_1}^{g_{1}}+L_{g_1}^{f_{i+1}} \big)
+
\Big(
\sum\limits_{j=1}^{k-1}L_{g_i}^{g_{i+1}}+L_{g_k}^{g_{1}}
\Big).
\end{aligned}
$$
The first bracket is equivalent to $L_{f_i}^{f_{i+1}}$ due to the first elementary move. The second bracket is zero, it corresponds to the equilibrium condition at the vertex around which the loop is taken (it is either $p_i$ or $q_i$).
\end{proof}

\section{Liftings for non-planar frameworks}
\label{Liftings of surfaces}

\subsection{Definition of lifting}

For a single-valued lifting we need the so-called {it monodromy conditions}: all the lifts along face-loops (having the same starting and end faces) are zero. If this is not the case one still can consider multivariate functions that will correspond to various coverings of surfaces.

\begin{definition}
Let $S$ be a polygonal surface in $\r^2$; 
let $F(S)$ be its framework; and let $w$ be a self-stress on $F(S)$.
We say that a self-stress $w$
is {\it monodromy-free}
if for any face $f$ of $S$ and any oriented face-loop $\gamma$ starting and ending at $f$ we have
$$
\tau_{\gamma,w}(f,f)=0.
$$
\end{definition}

According to Theorem~\ref{homotopic paths}
we can update the notation of lifts for monodromy-free stresses, since the lifts here are entirely the properties of the starting and the end faces, and there is no any dependency of the face-path choice.

\begin{definition}
Let $w$ be monodromy-free.
Set
$$
\tau_{w}(f,f')=\tau_{\gamma,w}(f,f'),
$$
for any oriented face-path $\gamma$ connecting $f$ and $f'$.
\end{definition}

Now let us we give a natural extension of the notion of lifting for
arbitrary frameworks (considered as 1-skeletons of some surfaces).

\begin{definition}
Consider a monodromy-free self-stress $w$ on $F(S)$.
Let $f$ be any face of $S$.
We say that the function
$$
\tau_{w,f}:S\to \r,
\quad \hbox{such that} \quad \tau_{w,f}(p)=\tau_{w}(f,f_p)(p),
$$
where $f_p$ is the face containing $p$,
is a {\it lifting} of $S$ with respect to $w,f$.
\end{definition}

\begin{figure}[t]
\includegraphics[height=4cm]{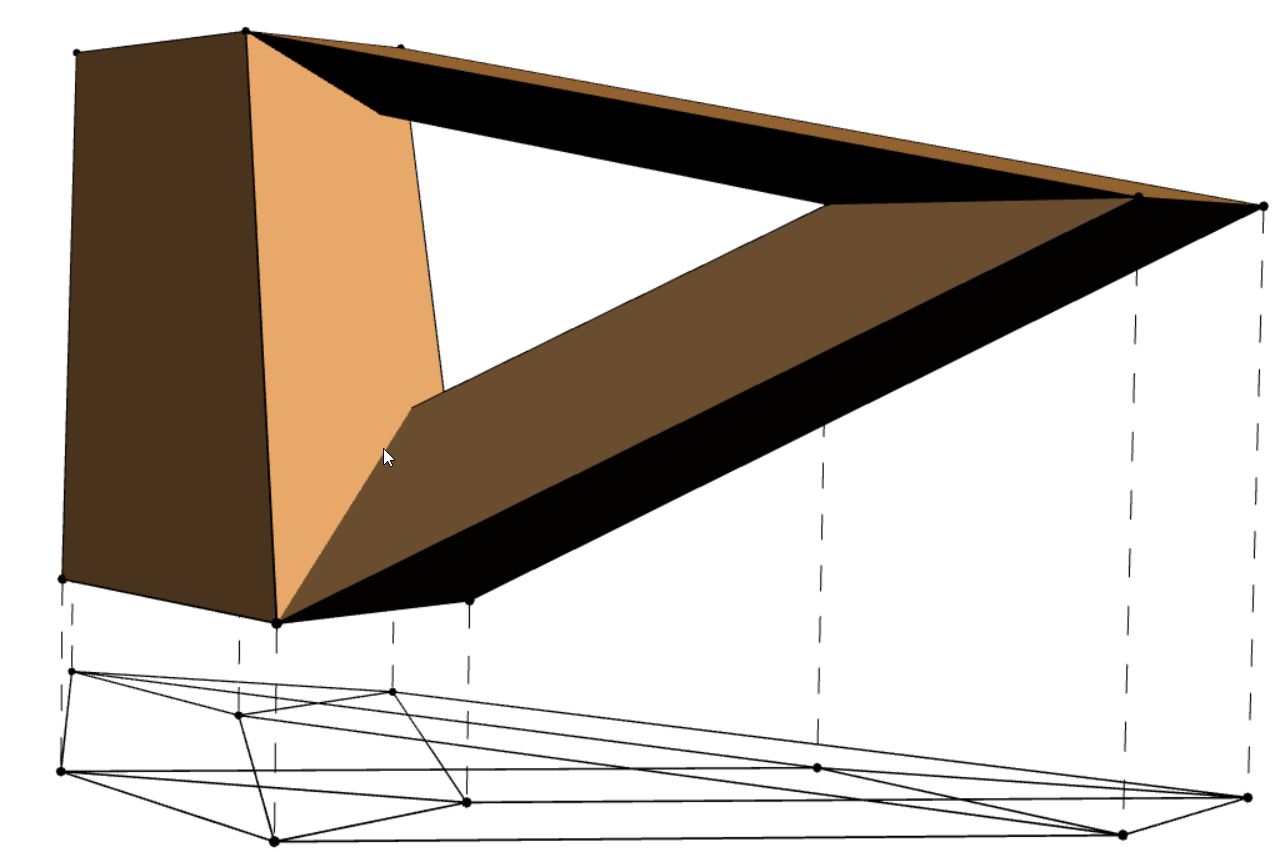}
\caption{Lifting of a framework, corresponding to a torus.
\label{torus.jpg-fig}}
\end{figure}

In Figure~\ref{torus.jpg-fig} we show an example of a monodromy-free lifting of a framework, corresponding to a torus.

\subsection{A few properties and remarks}

We start with the following important observation.

\begin{proposition}
The function $\tau_{w,f}$ is piecewise-linear; 
it is linear at every face of $S$.
\qed
\end{proposition}

\begin{remark}
In case if $w$ is not monodromy-free for $S$, one can consider any covering
of $S$ that has all the monodromies to be trivial (e.g. the universal covering).
\end{remark}

We arive to the following natural problem.

\begin{problem}
Study combinatorial conditions for a self-stress to be monodromy-free.
\end{problem}

The same question is actual for single face-loops.

\begin{remark}
Any polygonal surface immersion $S$ in $\r^3$ provides a self-stress on the framework obtained by projecting the 1-skeleton of $S$ to any plane along the direction orthogonal to the plane.
\end{remark}

\begin{remark}
Any polygonal surface immersion in $\r^3$ whose 1-skeleton projection is a covering of a framework provides a self-stress on that framework.
Here we sum up the values of stresses at all projection to the edges. Note that such self-stress might be zero at some or all edges.
\end{remark}

It is interesting to admit that different polygonal surfaces for the same graph result in different liftings.

\subsection{On monodromy-free liftings}

Let us start with the following general statement.

\begin{proposition}\label{finite-gen}
Let $S$ be a polygonal surface in $\r^2$.
Consider a lift with respect to an oriented face-loop $\gamma$ that represents a finite order element in the surface $S$.
\end{proposition}


\begin{proof}
Consider any closed oriented face-path $\gamma$.
Let the lift with respect to this face-path be the linear function $ax+by+c$. Then the lift of this face-path passed $n$ times (i.e., $n\gamma$) is  $nax+nby+nc$.

Assume now that $\gamma$ represents a finite order $m$ element in the surface $S$.
 Then $m\gamma$ is homotopic to a trivial face-path.
Hence the lift along $m\gamma$ is a zero function: $mx+my+mc=0$. Therefore, $ax+by+c=0$.
This concludes the proof.
\end{proof}

Recall that we assume that the boundaries of surfaces are homeomorphic to a disjoint collection of circles.

\begin{theorem}\label{betti-dimension}
Consider a compact connected polygonal surface~$S$ with boundary in $\r^2$ corresponding to a framework $F(S)$.
Let $b_1(S)$ be the Betty number of $S$.
Assume that the dimension of the space of stresses for $F(S)$ equals to $d$.
Let $d>3b_1(S)$.
Then there exists a non-trivial monodromy-free polyhedral lifting of $F(S)$.
\end{theorem}

\begin{remark}
For the case of oriented surfaces the statement of the theorem coincides with Theorem~45 in~\cite{cooper}.
\end{remark}

\begin{proof}
Every infinite generator of $H_1(S)$ provides three linear conditions on the space of self-stresses.
We can disregard finite generators due to Proposition~\ref{finite-gen} above.
Hence we have precisely $3b_1(S)$ linear conditions on the space of stresses, providing a solution in the case of $d>3b_1(S)$.
\end{proof}

\subsection{Liftings of surfaces with zero first Betti number}

In the case when $b_1(S)=0$, all the self-stresses are monodromy-free. So we have the following statement.

\begin{theorem}\label{t2-isomorphism}
Let $S$ be a polygonal surface with $b_1(S)=0$.
Then the function
$$
\omega \to \tau_{w,f}
$$
is a bijection between the space of all stresses on $F(S)$ and the space of all piecewise linear immersions of $S$ that projects $($by forgetting the last coordinates$)$ to $F(S)$ and having the face corresponding to $f$ to be in the plane $z=0$.
\end{theorem}

\begin{proof}
On the one hand, monodromy along any face-path  uniquely defines the lifting.

\vspace{2mm}

On the other hand any two adjacent faces of the lifting uniquely identify the stress of the projection on the edge between them.
\end{proof}

\begin{remark}
In particular it implies that polyhedra (that are homeomorphic polyhedra) in $\r^3$ are naturally enumerated by their projections to a plane with self-stresses equipped.  
\end{remark}

\begin{remark}
In case that  $b_1(S)>0$, one can cut the polyhedral surface along some edges to remove the infinite generators of $H_1(S)$. In particular we can remove all non-trivial loops and make the resulting surface simply-connected (we can treat it as {\it fundamental domain} of the surface).
For such fundamental domains, Theorem~\ref{t2-isomorphism} holds. 
Once we have a graph for a fundamental domain, we can uniquely reconstruct the multi-valued covering (lifting) for the original surface. 
\end{remark}

\begin{example}
Note that there are surfaces whose frameworks possess various topological types of lifting.
On Figure~\ref{4tori.jpg-fig} we show fundamental domains of four different liftings of the same torus with quadrangular faces.
\begin{figure}[t]
$$
\begin{array}{cc}
\includegraphics[height=3cm]{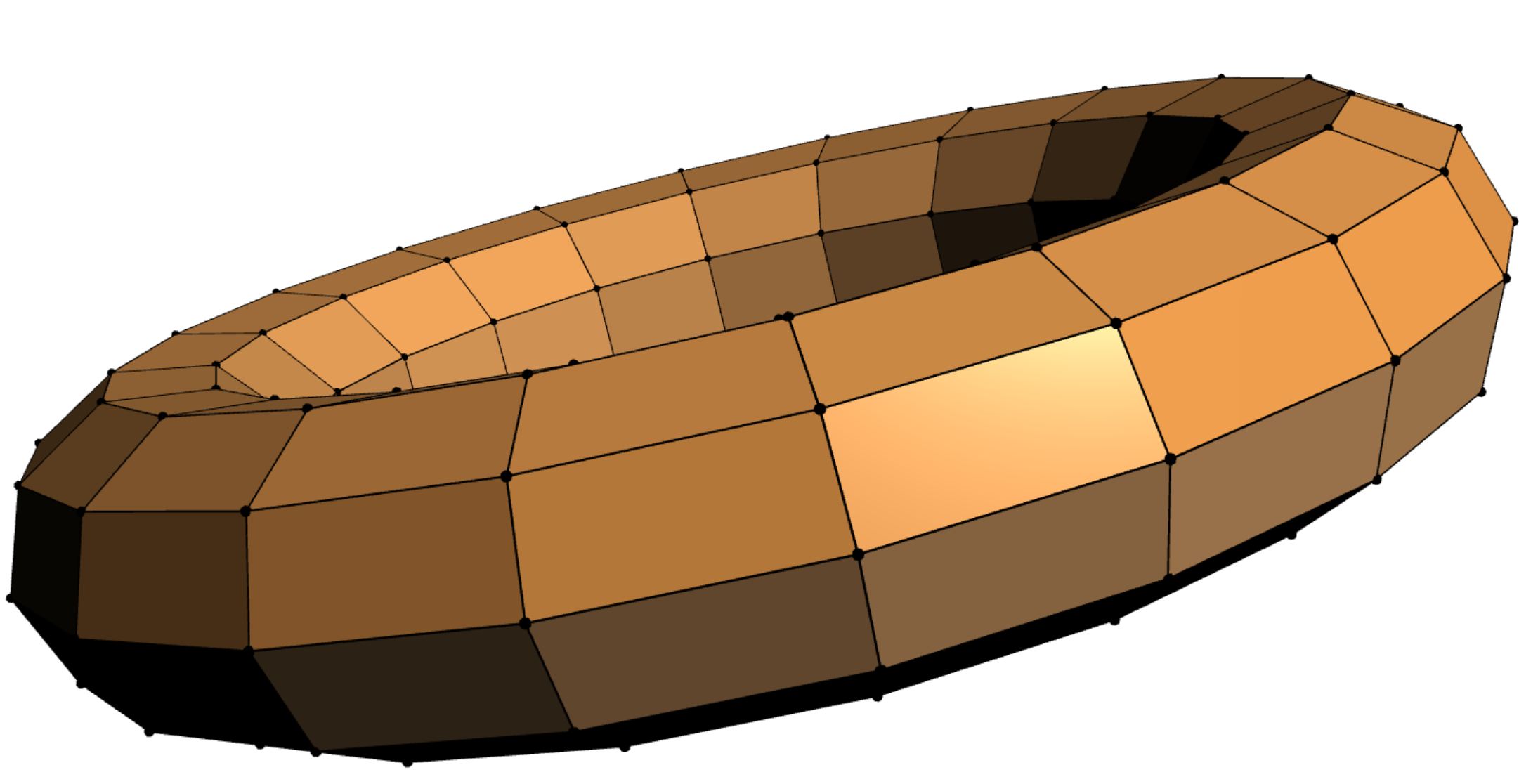}&
\includegraphics[height=4cm]{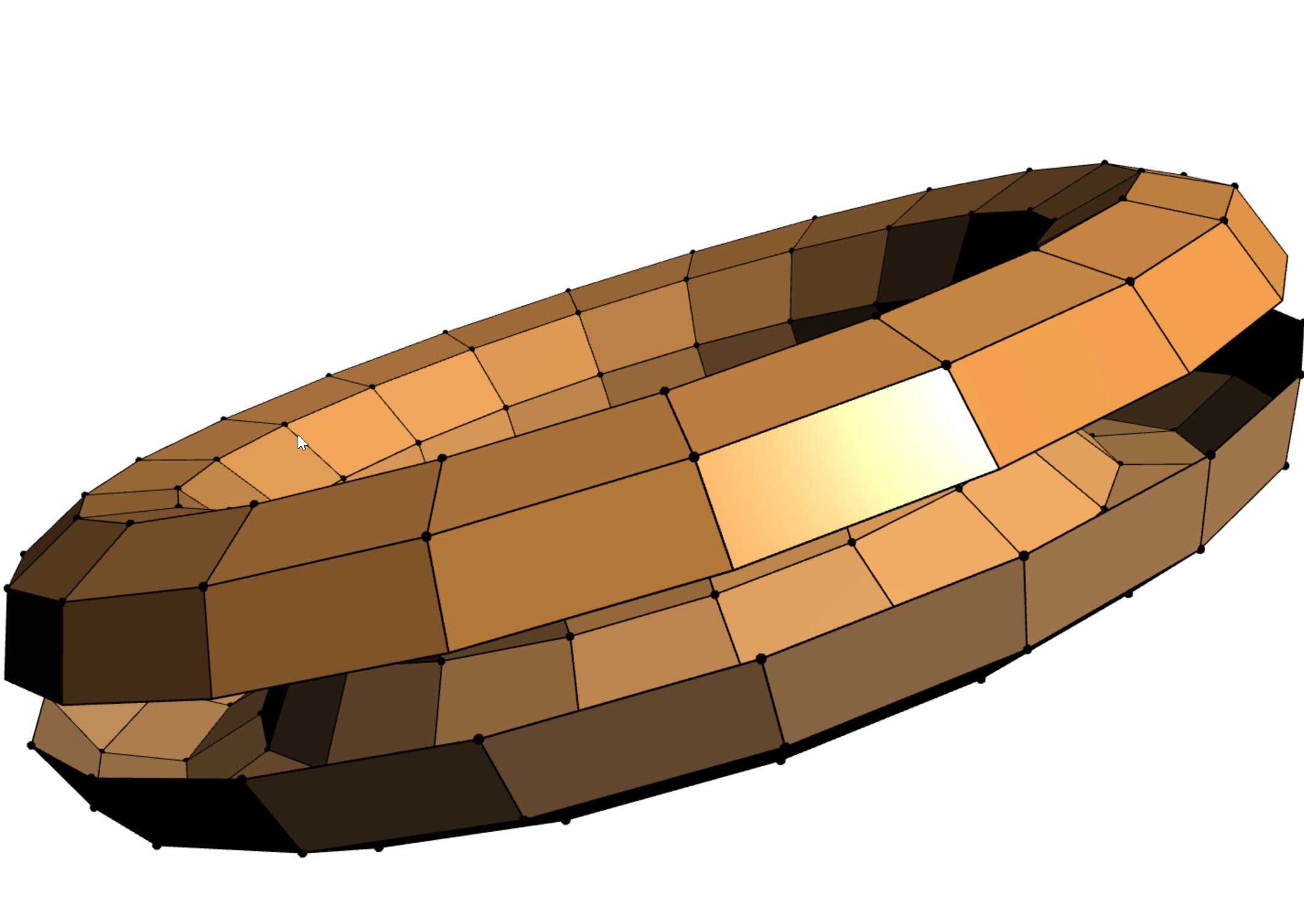}\\
\includegraphics[height=4cm]{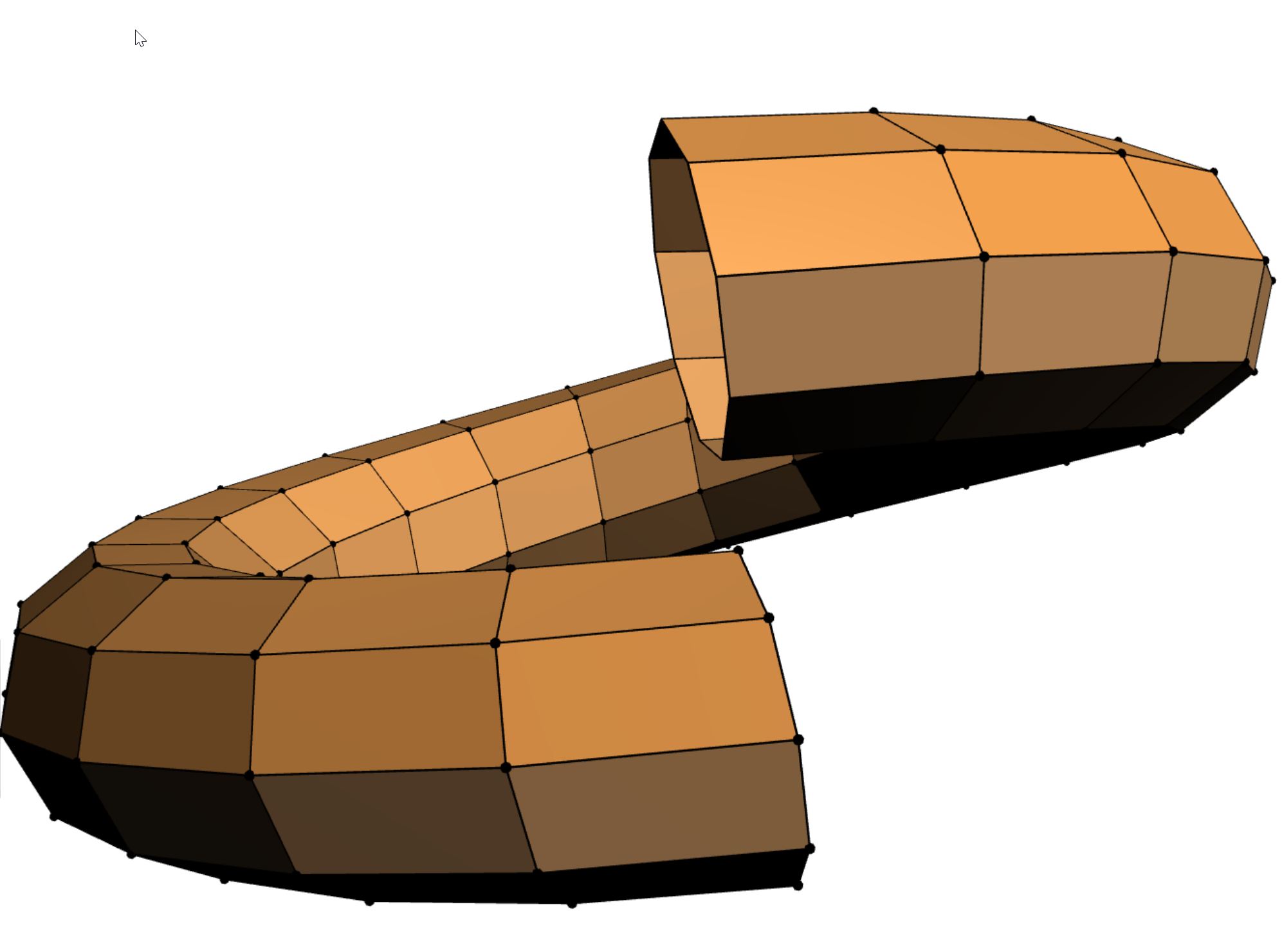}&
\includegraphics[height=4cm]{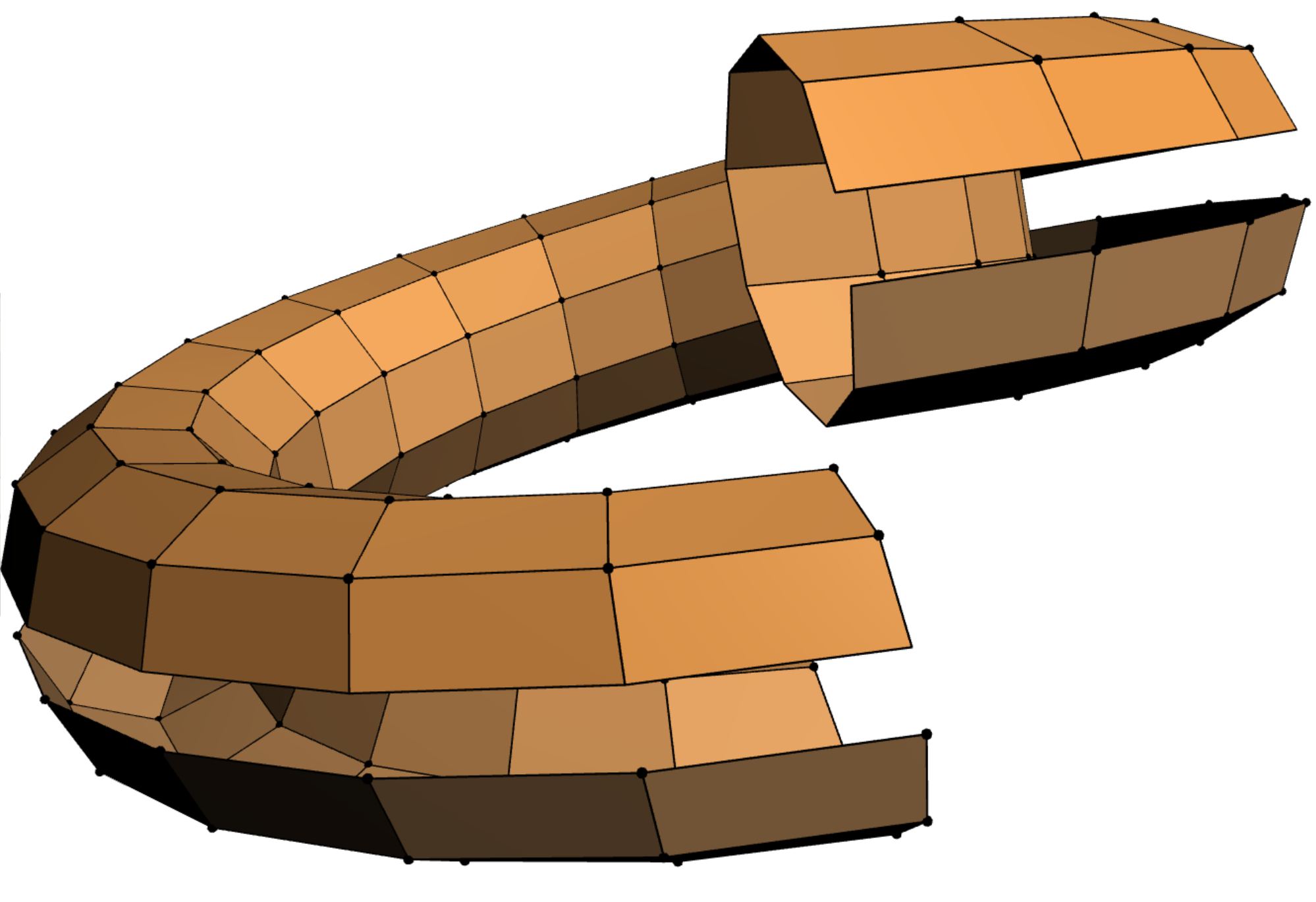}
\end{array}
$$
\caption{Different types of lifting of a framework corresponding to a torus.
\label{4tori.jpg-fig}}
\end{figure}
The monodromies of the upper-left torus are trivial.
Both monodromies of the lower-right torus are non-trivial. The remaining two tori have precisely one trivial monodromy.
\end{example}

\subsection{Particular example
}

Finally let us discuss an example of a lifting with one zero and one non-zero monodromy in details.

\begin{figure}[htb]
\includegraphics{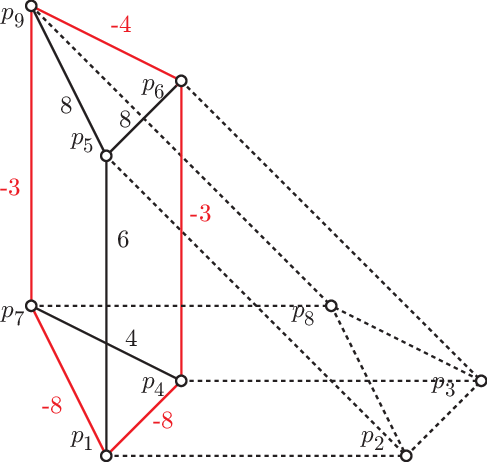}
\caption{Stresses on a triangular prism. The dotted lines are unstressed. The plane framework lifts to an infinite triangular prism
tube, corkscrewing out of the plane of the drawing.  The unstressed dotted edges are all horizontal, at levels 32 units apart from one another.
The stressed prism on the left lifts to a prism with two horizontal triangles, lifted $32$ units apart. \label{klein01fig}}
\end{figure}

\begin{example}
Consider the graph as on Figure~\ref{klein01fig} with vertices
$$
\begin{array}{ccc}
p_1=(0,0), & p_2=(4,0), & p_3=(5,1),\\
p_4=(1,1), & p_5=(0,4), & p_6=(1,5),\\
p_7=(-1,2), & p_8=(3,2), & p_9=(-1,6).
\end{array}
$$

Assume that we start with the face $f_0=p_1p_2p_3p_4$.
All the faces are as follows
$$
\begin{array}{ccc}
f_0=p_1p_2p_3p_4,& f_1=p_4p_7p_8p_3,& f_2=p_1p_7p_8p_2,\\
f_3=p_2p_5p_6p_3,& f_4=p_3p_8p_9p_6,& f_5=p_2p_8p_9p_5,\\
f_6=p_1p_4p_6p_5,& f_7=p_4p_7p_9p_6,& f_8=p_1p_5p_9p_7.\\
\end{array}
$$
Here the orders of vertices in faces are given with positive orientations.

Then we have
$$
\begin{array}{l}
\tau_{w,f_0}(f_0\cup f_1\cup f_2
\cup f_3\cup f_4\cup f_5)=\{32n|n\in \z\};
\\
\tau_{w,f_0}(f_6)=\{8x-8y+32n|n\in \z\};
\\
\tau_{w,f_0}(f_7)=\{-4x-8y+12+32n|n\in \z\};
\\
\tau_{w,f_0}(f_8)=\{-16x-8y+32n|n\in \z\}.
\\
\end{array}
$$

And we see that this lifting has one trivial monodromy along the meridian of the torus and one non-trivial monodromy along the longitude of the torus.

\end{example}


\bibliographystyle{plain}
    \bibliography{liftings}

\vspace{5mm}

\end{document}